\documentclass[11pt]{amsart}

\usepackage[utf8]{inputenc}
\usepackage{amsmath}
\usepackage{amssymb}
\usepackage{amsfonts}
\usepackage{amsthm}
\usepackage{mathtools}
\usepackage{thmtools}
\usepackage{enumerate}
\usepackage[top=2cm, bottom=2cm, left=2cm, right=2cm]{geometry}
\usepackage{graphicx}
\usepackage[all]{xy}
\usepackage{titlesec}
\usepackage{hyperref}
\hypersetup{
 colorlinks=true, 
 linkcolor=blue, 
 citecolor=blue, 
 filecolor=blue, 
 urlcolor=blue 
}
\usepackage{rotating}
\usepackage{xcolor}
\usepackage{setspace}
\usepackage[T1]{fontenc}
\usepackage{tikz}
\usepackage{comment}
\usepackage[all]{xy}
\usepackage[capitalize]{cleveref}

\titleformat{\section}
{\filcenter\large\bf} %format
{\thesection.\ } %label
{1pt} %sep
{} %

\titleformat{\subsection}[hang]
{\filcenter\bf}
{\thesubsection.}
{1pt}
{}

\declaretheoremstyle[bodyfont=\normalfont]{normalbody}

\declaretheorem[numberwithin=section,name=Theorem]{theorem}
\declaretheorem[sibling=theorem,style=normalbody,name=Definition]{definition}
\declaretheorem[sibling=theorem,name=Corollary]{corollary}
\declaretheorem[sibling=theorem,name=Lemma]{lemma}
\declaretheorem[sibling=theorem,name=Proposition]{proposition}

\declaretheorem[sibling=theorem,style=normalbody,name=Remark]{remark}

\newcommand{\supp}{\operatorname{supp}}
\newcommand{\Z}{\mathbb{Z}}
\newcommand{\N}{\mathbb{N}}

\newcommand{\R}{\mathbb{R}}
\newcommand{\C}{\mathbb{C}}

\renewcommand{\P}{\mathbb{P}}
\newcommand{\G}{\mathbb G}
\newcommand{\Ga}{\G_\mathrm{a}}

\newcommand{\Hom}{\operatorname{Hom}}
\newcommand{\Aut}{\operatorname{Aut}}

\newcommand{\PGL}{\mathrm{PGL}}

\newcommand{\Gm}{\mathbb{G}_{\mathrm{m}}}
\newcommand{\cone}{\mathrm{cone}}

\DeclareFontFamily{U}{wncy}{}
\DeclareFontShape{U}{wncy}{m}{n}{<->wncyr10}{}
\DeclareSymbolFont{mcy}{U}{wncy}{m}{n}
\DeclareMathSymbol{\Sh}{\mathord}{mcy}{"58}

\DeclareFontFamily{U}{wncy}{}
\DeclareFontShape{U}{wncy}{m}{n}{<->wncyr10}{}
\DeclareSymbolFont{mcy}{U}{wncy}{m}{n}
\DeclareMathSymbol{\Ch}{\mathord}{mcy}{"51}

\title{Complete toric varieties with semisimple automorphism group}

\author{Gabriel Barría Galland}
\address{Instituto de Matemáticas, Universidad de Talca, Chile}
\email{gabriel.barria@utalca.cl}

\begin{document}

\begin{abstract}
Let $X$ be a complete toric variety. We give a criterion to decide whether $X$ decomposes as a product of complete toric varieties by analyzing the $1$-skeleton of its fan. More precisely, we prove that any direct-sum decomposition of the $1$-skeleton induces a corresponding direct-sum decomposition of the fan itself. As an application, we show that if the identity component of the automorphism group is semisimple, then $X$ must be a product of projective spaces.
\end{abstract}

\maketitle

		\mbox{\hspace{1.6em} \textbf{ Keywords:} toric varieties, automorphisms groups}

\vspace{1em}
		\mbox{\hspace{1.6em} \textbf{ MSC codes (2020):} 14M25, 14J50} % \\
		\mbox{\hspace{1.3em} }

%
%
%		INTRODUCCION
%
%

\section*{Introduction}

Toric varieties form a well-studied class of algebraic varieties, described by a combinatorial structure known as a fan, see \cref{def:fan}. By definition, a toric variety $X$ contains an algebraic torus $(\mathbb{C}^*)^n$ as an open dense subset, and the torus action on itself extends to all of $X$. These varieties provide a rich family of examples at the intersection of algebraic geometry and combinatorics: their geometry is entirely encoded by the combinatorics of the associated fan, a feature that often renders explicit computations feasible. Classical examples include projective spaces and their products.

In this paper we are interested in the automorphism group of complete toric varieties. Notably, a toric variety $X$ always has at least the acting torus $T \simeq (\C^*)^n$ as a subgroup of $\Aut_X$, but $X$ may admit additional symmetries beyond this torus. In general, $\Aut(X)$ can be quite complicated, but a classical result of Matsumura and Oort \cite{MatOo} shows that it has the structure of a locally algebraic group. In particular, $\operatorname{Aut}^\circ(X)$ is itself a connected algebraic group. This fact means we can study $\Aut^\circ(X)$ using the tools of algebraic group theory.

In \cite{LucLieToric}, the authors show that if $X$ is a product of two complete toric varieties, then the automorphism group splits accordingly as a direct product, up to permutation of factors. This raises the following natural question: \emph{does a factorization of the identity component of the automorphism group of a toric variety force a corresponding factorization of the variety?}

To answer this question affirmatively under specific conditions, we first prove a result useful for decomposing a complete toric variety into products by examining the 1-skeleton of its fan:

\setcounter{section}{1}

\begin{theorem} \label{thm:main}
	Let \(V,W\) be \(\R\)-vector spaces and \(\Sigma\) a fan in $V\oplus W$ such that \(\Sigma(1)  = \Sigma(1)|_V \sqcup \Sigma(1)|_W\). Then \(\Sigma = \Sigma|_V \oplus \Sigma|_W\).
\end{theorem}

Using this theorem, we show that our original question is answered positively when \(\Aut_X^\circ \) is a semisimple algebraic group. In fact, Demazure proved that the automorphism group of a toric variety with a semisimple automorphism group is of the form \(\prod_{i=1}^k \PGL_{n_i}\) \cite{DemazureToric}. Furthermore, it is well known that if a toric variety has automorphism group \(\PGL_{n+1}\), then it is isomorphic to \( \P^n\), see \cref{cor 1}. Thus, using \cref{thm:main}, we provide an elementary proof of the following theorem:

\begin{theorem} \label{thm:application}
	Let $X$ be a complete toric variety. If $\Aut_X^{\circ}$ is semi-simple, then 
	\[
    X \simeq \prod_{i=1}^{k}\P^{n_i}.
    \]
\end{theorem}

This last theorem seems to be well-known to experts. A proof of this theorem using Lie theory has been communicated to us by Shintaro Kuroki on \cite{Shintaro}. We offer an alternative, elementary algebraic proof.

\medskip

\noindent \textbf{Acknowledgements. } I would like to thank Giancarlo Lucchini for his guidance on this work. I woulf also like to thank to Alvaro Liendo for his comments on the writing of this paper. This work was partially supported by ANID via beca Doctorado Nacional 2025 Folio 21251132.

\setcounter{section}{0}
\section{Preliminaries}

In this section we collect the required preliminaries on toric varieties and root systems.

\subsection{Preliminaries on Toric Varieties}

A toric variety \( X \) is a normal algebraic variety that contains a dense open subset isomorphic to the torus \( T \simeq (\C^*)^n \), where \( n \in \N \) and such that the action of the torus \( T \) on itself by translations extends to the entire toric variety \( X \). Let \( X \) and \( X' \) be two toric varieties with tori \( T \) and \( T' \), respectively. A toric morphism is a morphism of varieties \( \phi : X \to X' \) that restricts to a morphism of algebraic groups \( \varphi : T \to T' \).

Let \( M \) be a free \(\Z\)-module of finite rank, and let \( N \coloneqq \Hom_{\Z}(M, \Z) \) be its dual. We assign to \( M \) and \( N \) the vector spaces \( M_{\R} \coloneqq M \otimes_{\Z} \R \) and \( N_{\R} \coloneqq N \otimes_{\Z} \R \), respectively. We have the dual pairing \( \langle \cdot, \cdot \rangle : N_{\R} \times M_{\R} \to \R \) given by \( \langle p, m \rangle = p(m) \). Consider the algebraic torus \( T = N \otimes_{\Z} \Gm \), whose character lattice is \( M \) and whose lattice of one-parameter subgroups is \( N \). For each \( m \in M \), we denote by \( \chi^m : T \to \Gm \) the corresponding character, and for each \( p \in N \), we denote by \( \lambda_p : \Gm \to T \) the corresponding one-parameter subgroup.

A subset $\sigma$ of $N_\R$ is a convex polyhedral cone if there exists a finite set $S \subset N_\R$ such that
\[
\sigma = \cone(S) \coloneqq \left\{\sum_{s \in S}a_s s \mid a_s \in \R_{\geq 0}\right\}
\]
A convex polyhedral cone is pointed or strictly convex if it does not contain any non-zero vector subspaces of \(N_{\R}\). In this work, we restrict ourselves to pointed cones and will simply refer to them as cones. Let \(\sigma \subseteq N_{\R}\) be a cone. We say that \(\tau\) is a face of \(\sigma\) if it satisfies the following property: for every pair \(x, y \in \sigma\) such that \(x + y \in \tau\), it follows that \(x, y \in \tau\). We now state the definition of a fan, which will allow us to properly reference its properties.

\begin{definition} \label{def:fan}
	Let $V$ be an $\R$-vector space. A fan $\Sigma$ is a collection of cones $\sigma \subseteq V$ such that:
	\begin{enumerate}
		\item \label{def:fan a} Each cone $\sigma \in \Sigma$ is a pointed cone.
		
		\item \label{def:fan b} For every $\sigma \in \Sigma$, every face $\tau$ of $\sigma$ satisfies $\tau \in \Sigma$.
		
		\item \label{def:fan c} For every pair $\sigma_1,\sigma_2 \in \Sigma,$ the intersection $\sigma_1 \cap \sigma_2$ is a face of both $\sigma_1$ and $\sigma_2$.
	\end{enumerate}
	Additionally, if  $\bigcup_{\sigma \in \Sigma}\sigma = V$, we say that the fan is complete.
\end{definition}

Let $\Sigma$ and $\Sigma'$ be two fans in $N_\R$ and $N'_\R$ respectively. A morphism of fans is a linear transformation $\psi: N_\R \to N'_\R$ that restricts to a homomorphism $\psi':N\to N'$ such that, for every $\sigma \in \Sigma$, there exists $\sigma' \in \Sigma'$ with $\psi(\sigma) \subseteq \sigma'$. There is a categorical equivalence between fans and normal toric varieties with toric morphisms, which can be explored in detail in \cite{FultonToric} or \cite{CoxToric}.

Let $\Sigma$ be a fan in $N_\R$. The 1-skeleton of $\Sigma$ is defined as the set $\Sigma(1)$ of $1$-dimensional cones in $\Sigma$. Each such cone will be associated with a minimal generator $\rho \in N$, by abuse of notation, we identify $\rho $ with the cone.

\subsection{Preliminaries on Root Data and Root Systems}
Now, we present a definition from the theory of linear algebraic groups.

\begin{definition} \label{def:rootdatum}
	A root datum is a quadruple $\Psi = (X,R,X^{\ast},R^{\ast})$, where:
	\begin{enumerate}
		\item $X$ and $X^*$ are abelian groups of finite rank with a dual pairing $\langle \cdot , \cdot \rangle :X \times X^{\ast} \to \mathbb Z$.
		
		\item $R$ and $R^*$ are finite subsets of $X$ and $X^*$ respectively, with a bijection $\alpha \mapsto \alpha^*$ from $R$ to $R^*$. 
	\end{enumerate}
	
	For each $\alpha \in R$, we define the endomorphisms $s_\alpha$ and $s_{\alpha^*}$ by: 
    \begin{equation}\label{eq:1}
        s_\alpha(x) = x - \langle x,\alpha^*\rangle \alpha, \ \forall x\in X \quad \text{and} \quad s^*_\alpha(y) = y - \langle \alpha,y\rangle \alpha^* \ \forall y \in X^*
    \end{equation}

	satisfying the following axioms:
	\begin{enumerate}[(i)]
		\item \label{def:rootdatumI} For all $\alpha \in R, \langle \alpha, \alpha^* \rangle = 2$. 
		\item \label{def:rootdatumII} For all $\alpha \in R, s_\alpha(R) = R$ and $s_{\alpha^*}(R^*) = R^*$.
	\end{enumerate}
\end{definition}
	In the context of automorphism groups of toric varieties, the spaces $X$ and $X^*$ will correspond to the character lattice and cocharacter lattice of a torus.
	
	\begin{remark}\label{rmk:rootdata}
		Let $\Psi = (X,X^*,R,R^*)$ be a root datum. Note that $(X^*, X, R^*,R)$ is also a root datum and it is called the dual root datum of $\Psi$.
		
		We define $V \coloneqq \langle R \rangle \otimes_\Z \R$, where $\langle R \rangle$ is the subgroup generated by $R$ in $X$. If $R \neq \varnothing$,  then $R$ is a root system in $V$ (see \cite[Ch. VI, $\S$ 1, Definition 1]{BourbakiLie46}). That is, it satisfies the following axioms:
		\begin{enumerate}[(a)]
			\item \label{rmk:rootdataI}$R$ is finite, spans $V$, and $0 \notin R$. 
			
			\item \label{rmk:rootdataII} For each $\alpha \in R$ there exists $\alpha^* \in V^*$ such that $\langle \alpha , \alpha^* \rangle = 2$ and the endomorphism $s_\alpha$ from \cref{eq:1} stabilizes $R$. 
			\item \label{rmk:rootdataIII} For each $\alpha \in R,$ $\alpha^*(R) \subseteq \Z$. 
		\end{enumerate}
		
		As stated in \cite[Ch. VI, \S 1, Proposition 2]{BourbakiLie46}, $R^*$ is a root system in $V^*$ with the endomorphism $s_{\alpha^*} \coloneqq s_\alpha^*$.
	\end{remark}
	
	\begin{lemma}
		Let $T$ be an algebraic torus and $M$ the character lattice of $T$. If $R \subseteq M$ is a root system in $M_\R$, then there exists a unique subset $R^* \subseteq N \coloneqq \Hom_{\Z}(M,\Z)$ such that $(M,N,R,R^*)$ is a root datum.
	\end{lemma}
	
	\begin{proof}
		Let $\alpha \in R$. By \cite[Lemma 7.1.8]{SpringerLAG}, there exists a unique $\alpha^{\ast}$ in the dual space of $V$ satisfying $\langle \alpha, \alpha^* \rangle = 2$ and, for all $x \in M$,
		\[
		s_\alpha(x) = x - \langle x , \alpha^* \rangle \alpha.
		\]
		 Consequently, there exists a unique set of coroots $R^*$ fulfilling the conditions \ref{rmk:rootdataI}, \ref{rmk:rootdataII} and \ref{rmk:rootdataIII}. Thus, there is a bijection $R \to R^*$, and for each root $\alpha \in R$ the endomorphisms $s_\alpha$ and $s_\alpha^*$ satisfy the properties of \cref{def:rootdatum} (\ref{def:rootdatumI}) and (\ref{def:rootdatumII}) by property \ref{rmk:rootdataII}.
	\end{proof}

	\section{Proof of \cref{thm:main}}
	
	In this section we develop the necessary tools to prove \cref{thm:main}. Let $\Sigma$ be a fan in $N_\R$ and $W\subseteq N_\R$. For any subset $S$ of $\Sigma$, we use the following notation
	\[
	S|_W \coloneqq \{\sigma \cap W \mid \sigma \in S\}.
	\]
	We now state the following lemma:
	\begin{lemma}\label{lem:coneprojection}
		Let $V\oplus W$ be a real vector space and $\Sigma$ a fan in $V \oplus W$ such that $\Sigma(1) = \Sigma(1)_V \sqcup \Sigma(1)_W$. If $\sigma \in \Sigma$, then $\sigma\cap V$ and $\sigma \cap W$ are faces of $\sigma$. In particular, we have that $\Sigma \subseteq \Sigma|_V \oplus \Sigma|_W$. 
	\end{lemma}
	
	\begin{proof}
		Suppose that $x + y \in \sigma \cap V$, with $x,y \in \sigma$. We have that $x = v_x+w_x$ and $y=v_y+w_y$, where \(v_x, v_y \in \sigma \cap V\) and \(w_x, w_y \in \sigma \cap W\). Now, since we supposed that $x+y \in \sigma \cap V \subseteq V$, we deduce that $w_x+w_y = 0$, so $w_x = -w_y$. Because $\sigma$ is strictly convex we must have $w_x = w_y = 0$. Hence \(x, y \in V\), and since \(x, y \in \sigma\), it follows that \(x, y \in \sigma \cap V\), so $\sigma \cap V$ is a face of $\sigma$. Analogously we have that $\sigma \cap W$ is a face of $\sigma$. 
		
		We now prove that $\Sigma \subseteq \Sigma|_V \oplus \Sigma|_W$. To do this, we show that $\sigma \cap V \oplus \sigma \cap W = \sigma$. For $\sigma \in \Sigma$ note that $ \sigma \cap V \oplus \sigma \cap W \subseteq \sigma$ as $\sigma$ contains both intersections. Now, any $x \in \sigma$ can be written as a positive linear combination of the generators of $\sigma$
		\[
		x = \lambda_1 v_1 + \dots + \lambda_nv_n + \mu_1w_1+ \dots + \mu_mw_m,
		\]
		where the \(v_i\) are generators contained in \(V\) and the \(w_j\) are generators contained in \(W\). So $x$ can be written as an element of $\sigma\cap V \oplus \sigma\cap W$.
	\end{proof}
	
	We now continue studying what happens when the 1-skeleton of a complete fan can be separated into two nonempty sets contained in the components of a direct sum.
	
	\begin{lemma}\label{lem:projectioncomponent}
		Let \( V \oplus W \) be a real vector space. Suppose that \(\Sigma\) is a complete fan in \( V \oplus W \) such that \(\Sigma(1) = \Sigma(1)|_V \sqcup \Sigma(1)|_W\). Then \(\Sigma|_V\) and \(\Sigma|_W\) are complete fans in \( V \) and \( W \), respectively.
	\end{lemma}
	
	\begin{proof}
		We will prove that \(\Sigma|_W\) is a fan in \( W \); the proof for \(\Sigma|_V\) is analogous.  First we prove that \( \Sigma|_W \) satisfies condition (\ref{def:fan a}). Let \(\sigma\) be a cone in \(\Sigma\). Then \(\sigma = \text{cone}(b_1, \ldots, b_n)\) with generators \( b_i \in \Sigma(1) \) for all \( i = 1, \ldots, n \). Since \( \Sigma(1) = \Sigma(1)|_V \sqcup \Sigma(1)|_W \), the set \(\{b_1, \ldots, b_n\}\) admits a partition into disjoint subsets:
		
		\[\{w_1, \ldots, w_k\} \sqcup \{v_1, \ldots, v_\ell\},\]
		
		where \(\{w_1, \ldots, w_k\} \subseteq W\) and \(\{v_1, \ldots, v_\ell\} \subseteq V\). Consequently, \( \sigma \cap W = \text{cone}(w_1, \ldots, w_k)\), which is a convex polyhedral cone, establishing (\ref{def:fan a}).
		
		To verify condition (\ref{def:fan b}), let \(\tau\) be a face of \(\sigma \cap W\). By \cref{lem:coneprojection}, \(\sigma \cap W\) is a face of \(\sigma\). Since \(\Sigma\) satisfies (\ref{def:fan b}), \(\tau \in \Sigma\). Thus \(\tau \in \Sigma|_W\), proving (\ref{def:fan b}).
		
		For condition (\ref{def:fan c}), first observe that if \(\tau\) is a face of \(\sigma\), then for all \(x, y \in \sigma\) with \(x + y \in \tau\), we must have \(x, y \in \tau\). This property holds for any \(x, y \in \sigma \cap W\). Therefore, every face of \(\sigma\) contained in \(W\) determines a face of \(\sigma \cap W\). Now let \(\sigma'\) be another cone in \(\Sigma\). Since \(\Sigma\) satisfies (\ref{def:fan c}), \(\sigma \cap \sigma'\) is a common face of both \(\sigma\) and \(\sigma'\). It follows that \((\sigma \cap \sigma') \cap W\) is a face of both \(\sigma \cap W\) and \(\sigma' \cap W\), confirming (\ref{def:fan c}).
		
		Finally, we prove completeness. Suppose \(\Sigma|_W\) is not complete in \(W\). Then there exists some \(w \in W\) not contained in any cone of \(\Sigma|_W\).
		
		That is, for every \( \sigma \in \Sigma \), \( w \notin \sigma \cap W \). Since \( w \in W \), this implies \( w \notin \sigma \) for all \( \sigma \in \Sigma \), contradicting the completeness of \( \Sigma \). Therefore, \( \Sigma|_W \) is complete in \( W \).
	\end{proof}
	
	The following lemma provides the final ingredient needed to prove \cref{thm:main}. We show that for complete fans, containment implies equality.
	
	\begin{lemma}\label{lem:subfan}
		Let \( V \) be a real vector space, and let \(\Sigma\) and \(\Sigma'\) be two complete fans in \(V\) with \(\Sigma \subseteq \Sigma'\). Then \(\Sigma = \Sigma'\).
	\end{lemma}
	
	\begin{proof}
		 Consider any \(\sigma \in \Sigma'\). Since \(\Sigma\) is a complete fan, its support is the whole space. Let \(C \subseteq \Sigma\) be the set of cones \(\tau \in \Sigma\) such that \(\tau \cap \sigma \neq \{0\}\). Then:
        \[
        \sigma \subseteq \bigcup_{\tau \in C} \tau.
        \]
        As \(\Sigma \subseteq \Sigma'\), we have \(C \subseteq \Sigma'\). By \cref{def:fan}, each intersection \(\sigma \cap \tau\) is a face of \(\sigma\) and satisfies \(\sigma \cap \tau \in \Sigma\) for all \(\tau \in C\). Now observe that:
        \begin{equation} \label{Eq1}
            \sigma = \bigcup_{\tau \in C} (\sigma \cap \tau).
        \end{equation}
        Thus, \(\sigma\) is expressed as a union. If there exists \(\tau \in C\) such that \(\sigma \cap \tau = \sigma\), then \(\sigma\) is a face of \(\tau\) and consequently \(\sigma \in \Sigma\) by \cref{def:fan}.
        
        Otherwise, if \(\sigma \cap \tau \subsetneq \sigma\) for all \(\tau \in C\), then each \(\sigma \cap \tau\) is a proper face of \(\sigma\). By (\ref{Eq1}), every point in \(\sigma\) lies in some proper face. This contradicts \cite[Section 1.2, (7)]{FultonToric} since \(\sigma\) would have empty relative interior. Therefore this case is impossible, and we conclude that \(\Sigma = \Sigma'\).
	\end{proof}
	
	We now proceed with the main proof of this section:
	
	\begin{proof}[Proof of \cref{thm:main}]
		Let \( x \in V \oplus W \) be arbitrary. Then \( x = v + w \) with \( v \in V \) and \( w \in W \). By \cref{lem:projectioncomponent}, \( \Sigma|_V \) and \( \Sigma|_W \) are complete fans in \( V \) and \( W \) respectively. Thus there exist cones \( \sigma_V \in \Sigma|_V \) and \( \sigma_W \in \Sigma|_W \) such that \( v \in \sigma_V \) and \( w \in \sigma_W \). Consequently, \( x \in \sigma_V + \sigma_W \), proving that \( \Sigma|_V \oplus \Sigma|_W \) is complete. Since \cref{lem:coneprojection} gives \( \Sigma \subseteq \Sigma|_V \oplus \Sigma|_W \), and both are complete fans, \cref{lem:subfan} implies that \(\Sigma = \Sigma|_V \oplus \Sigma|_W\).
	\end{proof}
	
	Finally, we present a result useful for verifying that a fan is not complete.

	\begin{lemma} \label{lem:rayspan}
		Let \(\Sigma\) be a complete fan in an \(n\)-dimensional space. Then \(|\Sigma(1)| > n\).
	\end{lemma}
	
	\begin{proof}
		Suppose to the contrary that \(|\Sigma(1)| \leq n\). Let \(\Sigma(1) = \{v_1, \ldots, v_m\}\) with \(m \leq n\). If \(m < n\), the generators cannot span the \(n\)-dimensional space, contradicting completeness. If \(m = n\) and \(\{v_1, \ldots, v_n\}\) is linearly dependent, then \(\dim(\operatorname{Span}\{v_1, \ldots, v_n\}) < n\), leading to the same contradiction as before.
        
        Now consider the case where \(m = n\) and \(\{v_1, \ldots, v_n\}\) is linearly independent. We show that $v = -\sum_{i=1}^n v_i$ belongs to no cone in \(\Sigma\). Suppose there exists \(\sigma \in \Sigma\) containing \(v\). Then $v$ would be a non-negative linear combination of some generators:
        \[
        v = \sum_{i=1}^n \lambda_i v_i \quad \text{with} \quad \lambda_i \geq 0.
        \]
        But by construction, \(v\) is also the negative combination:
        \[
        v = -\sum_{i=1}^n v_i.
        \]
        Adding these expressions yields:
        \[
        \sum_{i=1}^n (1 + \lambda_i) v_i = 0.
        \]
        Since the generators are linearly independent, \(1 + \lambda_i = 0\) for all \(i\), implying \(\lambda_i = -1 < 0\), contradicting \(\lambda_i \geq 0\). Thus \(|\Sigma(1)| > n\).
	\end{proof}
	
	\section{Demazure Roots}
	
	Demazure roots were introduced by Demazure in \cite{DemazureToric}, who showed in that same article that this concept is fundamental for studying the automorphism group of complete toric varieties (we will be more precise about this in \cref{sec:3.2}). It thus becomes a crucial object for the development of this article and will be used repeatedly. Let \( T \) be a maximal torus in a toric variety \( X \), both of dimension \( n \). Recall that their character module \( M = \Hom(T, \Gm) \) and cocharacter module \( N = \Hom(M, \Z) \) are then free \(\Z\)-modules of rank \( n \).
	
	\begin{definition} \label{def:demazureroot}
		Let \(\Sigma\) be a fan in \(N_\mathbb{R}\). The set of Demazure roots of \(\Sigma\) is
		\[
		\mathcal R(\Sigma) := \left\{ \alpha \in M \mid \exists \rho \in \Sigma(1),\, \langle \rho, \alpha \rangle = -1 \ \text{and}\  \langle \rho', \alpha \rangle \geq 0,\  \forall \rho' \in \Sigma(1) \setminus \{\rho\} \right\}.
		\]
		Moreover, we denote by \(\rho_\alpha\) the unique element in \(\Sigma(1)\) satisfying \(\langle \rho_\alpha, \alpha \rangle = -1\).
	\end{definition}
	
	Note that Demazure roots depend solely on the 1-skeleton of the fan.
	
	\begin{remark}
		Demazure roots were originally defined in \cite{DemazureToric}. The original definition uses opposite signs compared to our convention. Readers consulting this reference should note this discrepancy.
	\end{remark}

	\subsection{Roots of the projective space}
	
	From now on, we denote by $I_n$ the set $\{1,\dots,n\}$ and by $I_n^{0}$ the set $\{0,1,\dots,n\}$. In a suitable basis, the standard fan of \(\P^n\) is formed by taking all strictly convex cones generated by subsets of its rays:
	\[
	\Sigma_{\P^n}(1) = \{e_i \mid i \in I_n\} \cup \left\{-\sum_{i=1}^n e_i\right\} \subseteq N \otimes_{\Z} \R.
	\]
	See \cite[Example 3.1.10]{CoxToric} for details. Following \textit{ibid.}, we henceforth fix \( e_0 := -\sum_{i=1}^n e_i \), so that \(\Sigma_{\P^n}(1) = \{e_i \mid i \in I_n^0\}\).
	
	We now compute the Demazure roots of \(\Sigma_{\P^n}\) by direct application of \cref{def:demazureroot}.
	
	\begin{proposition} \label{prop:projective_space_roots}
		Let \(\Sigma\) be a fan with \(\Sigma(1) = \{e_i \mid i \in I_n^0\} = \Sigma_{\P^n}(1)\). Then
		\begin{equation}
			\mathcal{R}(\Sigma) = \{\pm e_i^* \mid i \in I_n\} \cup \{e_{ij}^* \mid i,j \in I_n, i \neq j\},
		\end{equation}
		where \( e_{ij}^* := e_i^* - e_j^* \). In particular, this is the set of Demazure roots for \(\P^{n}\).
	\end{proposition}
	
	\begin{proof}
		Direct computation via \cref{def:demazureroot} yields:
		\[
		\{\pm e_i^* \mid i \in I_n\} \cup \{e_{ij}^* \mid i,j \in I_n, i \neq j\} \subseteq \mathcal{R}(\Sigma).
		\]
		For the reverse inclusion, let \(\alpha = \sum_{i=1}^n \lambda_i e_i^*\) be a Demazure root. By \cref{def:demazureroot}, there exists a unique \(\rho \in \Sigma(1)\) with \(\langle \rho, \alpha \rangle = -1\). 
		
		First consider \(\rho = e_j\) for some \(j \in I_n\). Then \(\langle e_j, \alpha \rangle = \lambda_j = -1\). Since \(\rho_\alpha = e_j\) is the unique ray satisfying \(\langle \rho_\alpha, \alpha \rangle < 0\), we must have \(\langle e_k, \alpha \rangle = \lambda_k \geq 0\) for all \(k \in I_n \setminus \{j\}\). 
		
		Additionally, the root condition requires \(\langle e_0, \alpha \rangle \geq 0\), implying:
		\[
		-\sum_{i=1}^n \lambda_i \geq 0.
		\]
		Thus \(\sum_{i \in I_n \setminus \{j\}} \lambda_i \leq 1\). Since \(\lambda_k \geq 0\) for all \(k \neq j\), at most one index \(s \in I_n \setminus \{j\}\) satisfies \(\lambda_s \neq 0\), and in that case \(\lambda_s = 1\). This shows that:
		\[
		\alpha \in \{-e_j^*\} \cup \{e_{sj}^* \mid s \in I_n \setminus \{j\}\} \subseteq \{\pm e_i^* \mid i \in I_n\} \cup \{e_{ij}^* \mid i,j \in I_n, i \neq j\}.
		\]
		Now consider \(\rho = e_0\). By \cref{def:demazureroot}, \(\langle e_j, \alpha \rangle \geq 0\) for all \(j \in I_n\), i.e., \(\lambda_j \geq 0\) for all \(j\). Then:
		\[
		\langle e_0, \alpha \rangle = -\sum_{i=1}^n \lambda_i = -1.
		\]
		Thus \(\sum_{i=1}^n \lambda_i = 1\). Since all \(\lambda_j \geq 0\), exactly one index \(s \in I_n\) satisfies \(\lambda_s = 1\) (others zero). Therefore \(\alpha = e_s^*\), which completes the inclusion.
	\end{proof}
	
	For the subsequent section, we denote by \( \mathcal{R}_n \) the set of Demazure roots of \(\Sigma_{\P^n}\). We have just characterized the roots of any fan sharing \(\P^n\)'s 1-skeleton. We now establish that for complete toric varieties, this set of rays uniquely determines the fan (up to a change of basis). In other words, we prove that \(\P^n\) is characterized by its 1-skeleton.
	
	\begin{proposition} \label{prop:projective_space_skeleton}
		Let \(X\) be a complete toric variety with fan \(\Sigma\). If \(\Sigma(1) = \{e_i \mid i \in I_n^0\}\), then \(X \simeq \P^n\).
	\end{proposition}
	
	\begin{proof}
		Let \(\sigma\) be any cone in \(\Sigma\), generated by a subset \(S \subseteq \Sigma(1)\). If \(S = \Sigma(1)\), then \(\sigma\) would contain the line \(\{\lambda e_0 \mid \lambda \in \R\}\) (since \(e_0 = -\sum_{i=1}^n e_i\)), violating strict convexity. Thus \(S \subsetneq \Sigma(1)\).
		
		By definition, \(\Sigma_{\P^n}\) consists precisely of all strictly convex cones generated by proper subsets of \(\{e_i \mid i \in I_n^0\}\). Therefore \(\sigma \in \Sigma_{\P^n}\), establishing \(\Sigma \subseteq \Sigma_{\P^n}\). Since both \(\Sigma\) and \(\Sigma_{\P^n}\) are complete fans, \cref{lem:subfan} implies \(\Sigma = \Sigma_{\P^n}\).
	\end{proof}
	
	\subsection{Relation Between Roots and Root Data}\label{sec:3.2}
	
	In this subsection, we connect the root datum of linear algebraic groups with the Demazure roots defined at the beginning of this section. Let \(\Sigma\) be a complete fan, \(X\) the corresponding toric variety, and \(\mathcal{R}(\Sigma)\) its set of Demazure roots. As mentioned in the introduction, \(\Aut_X\) is a locally algebraic group and \(\Aut_X^\circ\) is a connected algebraic group. Fixing a maximal torus \(T \subseteq \Aut_X\), we obtain the character module \(M \coloneqq \Hom(T, \G_m)\) and its dual \(N \coloneqq \Hom(M, \Z)\). How do Demazure roots of \(\Sigma\) relate to \(\Aut X\)? A partial answer is provided by \cite[Theorem 2]{LucLieToric}, which synthesizes key results from Demazure:
	\begin{enumerate}
		\item \cite[Section 4, \S 5, Theorem 3]{DemazureToric}
		\item \cite[Section 4, \S 6, Proposition 11]{DemazureToric}
	\end{enumerate}
	
	Note that Demazure's original theorems contain additional statements that complement the following result.
	
	\begin{theorem}\label{thm:demazure}
		Let \(\Sigma\) be a complete fan and \(X\) its corresponding toric variety. Denote by \(R(\Sigma)\) the set of Demazure roots of \(\Sigma\). For \(\alpha \in \mathcal R(\Sigma)\), let \(\rho_\alpha\) be the corresponding ray, and define the rational morphism:
		\[
		a_\alpha : \Ga \times T \to T \quad \text{given by} \quad (s, t) \mapsto t \cdot \lambda_{\rho_\alpha}\big(1 + s\chi^\alpha(t)\big).
		\]
		\begin{enumerate}
			\item \label{Demazure 1} The rational morphism \(a_\alpha\) defines a faithful \(\Ga\)-action on \(X\). In particular, the induced morphism \(t_\alpha : \Ga \to \Aut_X\) is a closed immersion.
			
			\item \label{Demazure 2} The image of \(t_\alpha\) is normalized by \(T\). Specifically:
			\[
			t^{-1} t_\alpha(s)  t = t_\alpha\big(\chi^\alpha(t)s\big) \quad \forall s \in \Ga,  \forall t \in T.
			\]
			\item \label{Demazure 3} A character \(\beta\) of \(T\) is a Demazure root of \(\Sigma\) if and only if \(a_\beta\) satisfies conditions (\ref{Demazure 1}) and (\ref{Demazure 2}).
			
			\item \label{Demazure 4} The connected component \(\Aut^\circ_X\) is generated by \(T\) and the images of \(t_\alpha\) for all \(\alpha \in \mathcal R(\Sigma)\). Consequently, \(\Aut_X\) is smooth.
			
			\item \label{Demazure 5} \(\Aut_X / \Aut^\circ_X\) is a finite constant algebraic group, isomorphic to the quotient of $\Aut(\Sigma)$ by the subgroup generated by morphisms of the form:
			\[
			\lambda \mapsto \lambda - \langle \lambda, \alpha \rangle (\rho_{-\alpha} - \rho_\alpha),
			\]
			where \(\alpha, -\alpha \in R(\Sigma)\).
		\end{enumerate}
	\end{theorem}
	
	Items (\ref{Demazure 1}), (\ref{Demazure 2}), (\ref{Demazure 3}), and (\ref{Demazure 4}) were originally proved in \cite[Section 4, \S 5, Theorem 3]{DemazureToric}. However, (\ref{Demazure 1}), (\ref{Demazure 4}), and part of (\ref{Demazure 5}) were reproved in \cite{LucLieToric} using the modern framework presented here. Item (\ref{Demazure 5}) is established in \cite[Section 4, \S 6, Proposition 11]{DemazureToric}. When consulting Demazure's work, readers should note the sign convention difference for \(\rho_\alpha\) and \(\rho_{-\alpha}\). The following theorem will facilitate part of the proof of our next result.
	
	\begin{lemma} \label{lem:rootdatum_demazureroots}
		Let \(X\) be a complete toric variety with fan \(\Sigma\) and \(G = \Aut X\) its automorphism group. Let \(T\) be a maximal torus of \(G\). Then there exists a natural correspondence between Demazure roots and roots of the adjoint action of \(T\) on the Lie algebra \(\mathfrak{g}\) of \(G\).
		
		In particular, if \(G^\circ\) is reductive, the Demazure roots naturally identify with roots in the associated root datum. That is, if \(\Psi(M, \Delta)\) is the root datum of \(G\), then \(\Delta = \mathcal{R}(\Sigma)\).
	\end{lemma}
	
	\begin{proof}
		By \cite[Lemma 7.3.3]{SpringerLAG} and the correspondence between Lie subalgebras and connected subgroups \cite[\S 13.1, Theorem]{HumphLAG}, there exists a bijection between roots \(\alpha\) of the adjoint action of \(T\) on \(\mathfrak{g}\) and closed immersions \(\iota_\alpha : \Ga \to G\) satisfying:
		\[
		t \iota_\alpha(s) t^{-1} = \iota_\alpha(\chi^\alpha(t)s) \quad \forall t \in T,  \forall s \in \Ga.
		\]
		Conversely, \cref{thm:demazure} establishes that Demazure roots correspond bijectively to closed immersions \(t_\beta : \Ga \to G\) defined by:
		\[
		t \cdot t_\beta(s) \cdot t^{-1} = t_\beta(\chi^\beta(t)^{-1}s).
		\]
		Note the sign inversion in the exponent compared to the adjoint action correspondence.
		
		Since \(\chi^\beta(t)^{-1} = \chi^{-\beta}(t)\), this implies \(\beta \leftrightarrow -\alpha\) under the correspondence. Thus we obtain:
		\[
		\text{Adjoint roots} \ \alpha \quad \leftrightarrow \quad \text{Demazure roots} \ -\alpha.
		\]
		
		When \(G^\circ\) is reductive, the root system satisfies \(\alpha \in \Delta \iff -\alpha \in \Delta\). Therefore the sign reversal becomes irrelevant, yielding \(\Delta = \mathcal{R}(\Sigma)\).
	\end{proof}
	
\section[Semisimple automorphism group]{Toric varieties with semisimple automorphism group}
	The objective of this section is to apply the framework developed in Sections 2 and 3 to prove the main application of the article (\cref{thm:application}). 
	
	\subsection{The case of the projective linear group}
	Fix \( n \in \mathbb{N} \). In this section, we prove that if a fan has the root system of \(\operatorname{PGL}_{n+1}\), then it must be the fan of \(\P^n\). We achieve this by classifying complete fans whose Demazure root set equals \(\mathcal{R}_n\).
	
	Our objective is to determine necessary conditions on the rays of a complete fan that yield the root system \(\mathcal{R}_n\). We adopt a root-centric approach: For each fixed root \(\alpha \in \mathcal{R}_n\), we characterize coroots \(\rho\) that are uniquely associated to \(\alpha\) and cannot correspond to other roots in \(\mathcal{R}_n\).
	
	After determining all possible coroot assignments, we synthesize this information to reconstruct all 1-skeletons compatible with \(\mathcal{R}_n\). The foundational step is understanding how the entire root system constrains candidate rays \(\rho_\alpha\). These constraints are formalized in the following definition.
	
	\begin{definition}\label{def:rayset}
		For \(\alpha \in \mathcal R_n\), define
		\[
		\mathcal{P}(\alpha) \coloneqq \left\{ p \in N \mid \langle p, \alpha \rangle = -1 \ \text{and}\  \langle p, \alpha' \rangle \geq -1,  \forall \alpha' \in \mathcal R_n \setminus \{\alpha\} \right\}.
		\]
	\end{definition}
	Note that \cref{def:demazureroot} associates each root with a unique ray. However, the proof of \cref{prop:projective_space_roots} shows that for \(n > 1\), every ray \(p\) in \(\Sigma_{\P^n}(1)\) satisfies:
	\[
	\left|\left\{ \alpha \in \mathcal R_n \mid \langle p, \alpha \rangle = -1 \right\}\right| > 1.
	\]
	Furthermore, \cref{def:demazureroot} implies that for every root \(\alpha \in \mathcal R_n\) and every ray \(p \in \Sigma_{\P^n}(1)\), \(\langle p, \alpha \rangle \geq -1\). The following proposition provides an explicit description of \(\mathcal{P}(\alpha)\).
	
	\begin{proposition}\label{prop:rayset}
		Let \(\alpha \in \mathcal R_n\) and \(\epsilon \in \{-1, 1\}\).
		\begin{enumerate}
			\item If \(\alpha = \epsilon e_i^*\), then  
			\[
			\mathcal{P}(\alpha) = \left\{ \sum_{k=1}^n a_k e_k \mid a_i = -\epsilon \ \text{and}\ a_k \in \{0, -\epsilon\} \ \forall k \in I_n \setminus \{i\} \right\}.
			\]
			
			\item If \(\alpha = e_{ij}^*\), then  
			{ 
			\[
			\mathcal{P}(\alpha) = \left\{ \sum_{k=1}^n a_k e_k \mid a_i \in \{0, -1\},\ a_j = 1 + a_i,\ \text{and}\ a_k \in \{0, a_i, a_j\} \ \forall k \in I_n \setminus \{i,j\} \right\}.
			\]}
		\end{enumerate}
	\end{proposition}
	
	\begin{remark} \label{rmk a}
		For any \(\alpha \in \mathcal R_n\), either \(\alpha = e_i^*\) for some \(i \in I_n\) or \(\alpha = e_{ij}^*\) for some pair \(i,j \in I_n\). In either cases, any element \(\rho\) in the right-hand set satisfies  
		\[\rho = \sum_{k=1}^n a_k e_k \quad \text{with} \quad a_k \in \{0, \epsilon\} \ \text{and} \ \epsilon \in \{\pm 1\}.\]
	\end{remark}
	
	\begin{proof}
		We first prove that the sets on the right-hand side are contained in those on the left. Let \(\alpha \in \mathcal R_n\). Suppose \(\alpha = \epsilon e_i^*\) for some \(i \in I_n\). If \(\rho = \sum_{k=1}^n a_k e_k\) with \(a_i = -\epsilon\) and \(a_k \in \{0, -\epsilon\}\), then \(\langle \rho, \alpha \rangle = -\epsilon^2 = -1\). Now consider \(\alpha' \in\mathcal R_n \setminus \{\alpha\}\):
		\begin{enumerate}
			\item If \(\alpha' = \delta e_l^*\) for some \(l \in I_n\) and \(\delta \in \{-1, 1\}\), then \(\langle \rho, \alpha' \rangle = \delta a_l\).
			
			\item If \(\alpha' = e_{lm}^*\) for some \(l, m \in I_n\), then \(\langle \rho, \alpha' \rangle = a_l - a_m\).
		\end{enumerate}
		By \cref{rmk a}, in both cases \(\langle \rho, \alpha' \rangle \in \{0, \pm 1\}\). Thus \(\rho \in \mathcal{P}(\alpha)\).
		
		Now suppose \(\alpha = e_{ij}^*\) for some \(i, j \in I_n\). If \(\rho = \sum_{k=1}^n a_k e_k\) with \(a_i = -1\), \(a_j = 0\), and \(a_k \in \{0, -1\}\), then \(\langle \rho, \alpha \rangle = -1\). For \(\alpha' \in \mathcal R_n \setminus \{\alpha\}\), note that \(\rho\) is a special case of the previous calculation, so \(\langle \rho, \alpha' \rangle \in \{0, \pm 1\}\). Thus \(\rho \in \mathcal{P}(\alpha)\). The calculation is analogous for \(\rho = \sum_{k=1}^n a_k e_k\) with \(a_i = 0\), \(a_j = 1\), and \(a_k \in \{0, 1\}\). Hence \(\rho \in \mathcal{P}(\alpha)\).
		
		We now prove the reverse inclusion. Let \(\alpha\) be a root in \(\mathcal R_n\) and \(\rho = \sum_{k=1}^n a_k e_k \in \mathcal{P}(\alpha)\). Since \(\epsilon e_i^* \in \mathcal R_n\) for all \(i \in I_n\) and \(\epsilon \in \{-1, 1\}\), \cref{def:rayset} implies that for every \(i \in I_n\) and \(\epsilon \in \{-1,1\}\):
		\[
		\langle \rho, \epsilon e_i^* \rangle = \epsilon a_i \geq -1.
		\]
		Therefore \(a_i \in \{0, \pm 1\}\) for all \(i \in I_n\). Now consider \(i, j \in I_n\) with \(i \neq j\). Since \(e_{ij}^* \in \mathcal R_n\), \cref{def:rayset} gives:
		\[
		\langle \rho, e_{ij}^* \rangle = a_i - a_j \geq -1.
		\]
		If \(a_j = -a_i\) with both non-zero, then \(\langle \rho, e_{ij}^* \rangle = -2\) (when \(a_i = 1\)) or \(\langle \rho, e_{ij}^* \rangle = 2\) (when \(a_i = -1\)). In the latter case, \(\langle \rho, e_{ji}^* \rangle = -2\). Both scenarios contradict \(\rho \in \mathcal{P}(\alpha)\). Thus, for any pair \(i, j \in I_n\), non-zero coefficients cannot have opposite signs.
		
		If \(\alpha = \epsilon e_i^*\) for some \(i \in I_n\), then:
		\[
		-1 = \langle \rho, \alpha \rangle = \epsilon a_i \implies a_i = -\epsilon.
		\]
		Since all non-zero coefficients must share the sign of \(a_i\) (i.e., \(-\epsilon\)), we have \(a_k \in \{0, -\epsilon\}\) for all \(k \in I_n\). Therefore, \(\rho\) has the form:
		\[
		\rho = \sum_{k=1}^n a_k e_k, \text{ with } a_i = -\epsilon \text{ and } a_k \in \{0, -\epsilon\} \text{ for } k \in I_n.
		\]
		This proves the set equality for the case \(\alpha = \epsilon e_i^*\). Now suppose \(\alpha = e_{ij}^*\) for some pair \(i, j \in I_n\). Then:
		\[
		\langle \rho, \alpha \rangle = a_i - a_j = -1.
		\]
		
		By previous results, \(a_i, a_j \in \{0, \pm 1\}\) and they share the same sign. Thus only two solutions satisfy this equation:
		\[
		a_i = -1 \ \text{and}\ a_j = 0, \quad \text{or} \quad a_i = 0 \ \text{and}\ a_j = 1.
		\]
		
		If \(a_i = -1\), then \(a_j = 0\) and:
		\[
		\rho = \sum_{k=1}^n a_k e_k \quad \text{with} \quad a_i = -1,\ a_j = 0,\ \text{and}\ a_k \in \{0, -1\} = \{0, a_i\}\ \text{for}\ k \in I_n \setminus \{i,j\}.
		\]
		
		If \(a_i = 0\), then \(a_j = 1\) and:
		\[
		\rho = \sum_{k=1}^n a_k e_k \quad \text{with} \quad a_i = 0,\ a_j = 1,\ \text{and}\ a_k \in \{0, 1\} = \{0, a_j\}\ \text{for}\ k \in I_n \setminus \{i,j\}.
		\]
		
		In both cases, \(\rho\) satisfies:
		\[
		\rho = \sum_{k=1}^n a_k e_k \quad \text{with} \quad a_i \in \{0, -1\},\ a_j = 1 + a_i,\ \text{and}\ a_k \in \{0, a_i, a_j\}\ \text{for}\ k \in I_n \setminus \{i,j\}.
		\]
		
		This establishes the set equality for the case \(\alpha = e_{ij}^*\).
	\end{proof}
	
	Let \(\alpha \in \mathcal R_n\). To construct the 1-skeleton of a fan with root system \(\mathcal R_n\), we select a ray \(\rho_\alpha\) from the set \(\mathcal{P}(\alpha)\). Fixing this ray forces other roots to share \(\rho_\alpha\). The set of roots associated to a fixed ray is defined as follows:
	
	\begin{definition} \label{def:rootset}
		For \(\rho \in N\), define
		\[\mathcal R(\rho) \coloneqq \{\alpha \in \mathcal R_n \mid \langle \rho, \alpha \rangle = -1\}.\]
	\end{definition}
	
	The following proposition computes the elements of \(\mathcal R(\rho)\) and provides a cardinality formula.
	
	\begin{proposition} \label{prop:rootset}
		Let \(\alpha \in \mathcal R_n\). Suppose \(\Sigma\) is a complete fan with \(R(\Sigma) = \mathcal R_n\), and let \(\rho_\alpha \in \Sigma(1)\) be the ray associated to root \(\alpha\). If \(\rho_\alpha = \sum_{k=1}^n a_k e_k\) with \(a_k \in \{0, \epsilon\}\) and \(a_i = \epsilon \in \{\pm 1\}\), then
		\[
		\mathcal R(\rho_\alpha) = \bigcup_{j \in \supp(\rho_\alpha)} \left\{ (\epsilon - a_k) e_k^* - \epsilon e_j^* \mid k \in I_n \right\}.
		\]
		In particular, the cardinality of \(\mathcal R(\rho_\alpha)\) satisfies:
		\[
		|\mathcal R(\rho_\alpha)| = |\supp(\rho_\alpha)| \cdot (n + 1 - |\supp(\rho_\alpha)|).
		\]
	\end{proposition}
	
	\begin{proof}
		Let \(\alpha'\) belong to the right-hand set, i.e.,
		\[
		\alpha' = (\epsilon - a_k)e_k^* - \epsilon e_j^*
		\]
		for some \(j \in \supp(\rho_\alpha)\) and \(k \in I_n\). Since \(j \in \supp(\rho_\alpha)\), \(a_j = \epsilon\), and a direct computation gives \(\langle \rho_\alpha, \alpha' \rangle = -\epsilon^2 = -1\). Thus:
		\[
		\mathcal R(\rho_\alpha) \supseteq \bigcup_{j \in \supp(\rho_\alpha)} \left\{ (\epsilon - a_k)e_k^* - \epsilon e_j^* \mid k \in I_n \right\}.
		\]
		
		To simplify the reverse inclusion, first assume \(\epsilon = 1\), so \(a_j = 1\) for all \(j \in \supp(\rho_\alpha)\). The case \(\epsilon = -1\) follows analogously. Let \(\alpha' \in \mathcal R_n\) satisfy \(\langle \rho_\alpha, \alpha' \rangle = -1\). We prove:
		\begin{equation} \label{Eq 4.1}
			\alpha' = (1 - a_k)e_k^* - e_j^* \quad \text{for some} \quad j \in \supp(\rho_\alpha),  k \in I_n. 
		\end{equation}
		Case 1: If \(\alpha' = \delta e_l^*\) with \(\delta \in \{-1, 1\}\), then \(-1 = \langle \rho_\alpha, \alpha' \rangle = \delta a_l\). Since \(\epsilon = 1\), \(a_l = 1\) and \(\delta = -1\). Thus \(\alpha' = -e_l^*\) for some \(l \in \supp(\rho_\alpha)\), satisfying (\ref{Eq 4.1}) with \(j = k = l\).
		
		Case 2: If \(\alpha' = e_{st}^*\) for \(s, t \in I_n\), then \(-1 = \langle \rho_\alpha, \alpha' \rangle = a_s - a_t\). With \(\epsilon = 1\), \(a_k \in \{0, 1\}\), so \(a_s = 0\) and \(a_t = 1\). Thus \(\alpha'\) satisfies (\ref{Eq 4.1}) with \(j = t \in \supp(\rho_\alpha)\) and \(s \in I_n \setminus \supp(\rho_\alpha)\). As these cover all possible \(\alpha'\), we conclude:
		\[
		\mathcal R(\rho_\alpha) = \bigcup_{j \in \supp(\rho_\alpha)} \left\{ (\epsilon - a_k)e_k^* - \epsilon e_j^* \mid k \in I_n \right\}.
		\]
		
		We now prove the cardinality formula. First, we claim that for distinct \(j, j' \in \supp(\rho_\alpha)\):
		\begin{equation} \label{Eq 4.2}
			\left\{ (\epsilon - a_k)e_k^* - \epsilon e_j^* \mid k \in I_n \right\} \cap \left\{ (\epsilon - a_k)e_k^* - \epsilon e_{j'}^* \mid k \in I_n \right\} = \varnothing.
		\end{equation}
		Suppose equality holds for some \(j \neq j'\) and \(k, k' \in I_n\):
		\[
		(\epsilon - a_k)e_k^* - \epsilon e_j^* = (\epsilon - a_{k'})e_{k'}^* - \epsilon e_{j'}^*.
		\]
		If \(k, k' \in \supp(\rho_\alpha)\), then \(\epsilon - a_k = 0\) and \(\epsilon - a_{k'} = 0\), implying:
		\[
		-\epsilon e_j^* = -\epsilon e_{j'}^*.
		\]
		This implies \( j = j' \). Now suppose \( k \in \supp(\rho_\alpha) \) and \( k' \notin \supp(\rho_\alpha) \). Then \( j' \neq k' \) (since \( k' \notin \supp(\rho_\alpha) \)). By linear independence of \(\{e_{k'}^*, e_{j'}^*\}\), we must have \( j = j' \) or \( j = k' \). If \( j = j' \), this contradicts our assumption \( j \neq j' \). If \( j = k' \), then:
		\[
		-\epsilon = \epsilon - a_{k'} \implies a_{k'} = 2\epsilon,
		\]
		contradicting \( a_{k'} \in \{0, \epsilon\} \). A symmetric argument applies when \( k' \in \supp(\rho_\alpha) \) and \( k \notin \supp(\rho_\alpha) \).
		
		Finally, assume \( k, k' \notin \supp(\rho_\alpha) \). By linear independence of the dual basis, possible equalities require:
        \begin{enumerate}
            \item  \(j = j'\) and  \(k = k'\)  or
            \item \(j = k'\) and \(k = j'\).
        \end{enumerate}
		In the latter case:
		\[
		-\epsilon = \epsilon - a_k \implies a_k = 2\epsilon,
		\]
		again contradicting \( a_k \in \{0, \epsilon\} \). Thus \( j = j' \) in all cases, proving (\ref{Eq 4.2}).
		
		Moreover, for any fixed \( j \in \supp(\rho_\alpha) \), the set:
		\[
		\left\{ (\epsilon - a_k)e_k^* - \epsilon e_j^* \mid k \in I_n \right\}
		\]
		has cardinality \( n + 1 - |\supp(\rho_\alpha)| \) because:

        \begin{enumerate}
            \item When \( k \in \supp(\rho_\alpha) \), \( \epsilon - a_k = 0 \), yielding identical roots \( -\epsilon e_j^* \)

            \item Distinct \( k \notin \supp(\rho_\alpha) \) produce distinct roots
        \end{enumerate}
		
		Thus each set has size \( n - |\supp(\rho_\alpha)| + 1 \). By (\ref{Eq 4.2}), the sets are disjoint across \( j \), so:
		\[
		|\mathcal R(\rho_\alpha)| = \sum_{j \in \supp(\rho_\alpha)} (n + 1 - |\supp(\rho_\alpha)|) = |\supp(\rho_\alpha)| \cdot (n + 1 - |\supp(\rho_\alpha)|).
		\]
	\end{proof}
	
	\begin{remark} \label{rmk:rayset}
		For any \(\rho \in N\),
		\[
		|\mathcal R(\rho)| = |\supp(\rho)| \cdot (n + 1 - |\supp(\rho)|) \geq n.
		\]
		Equality holds when \(|\supp(\rho)| \in \{1, n\}\). In particular, \(\rho\) has exactly \(n\) associated roots if and only if \(\rho = \pm e_i\) for some \(i \in I_n^0\).
	\end{remark}
	At this stage, we have gathered sufficient information to construct fans whose 1-skeleton is compatible with the root set \(\mathcal R_n\). As anticipated in this chapter's introduction, the resulting 1-skeletons are isomorphic to \(\Sigma_{\P^n}\). Adding the completeness hypothesis yields this subsection's main objective.
	
	\begin{theorem} \label{thm:roots_determine_projective_space}
		Let \(\Sigma\) be a complete fan with \(\mathcal{R}(\Sigma) = \mathcal R_n\). Then:
		\[
		\Sigma(1) = \Sigma_{\P^n}(1) \quad \text{or} \quad \Sigma(1) = -\Sigma_{\P^n}(1).
		\]
	\end{theorem}
	
	\begin{proof}
		Suppose \(\Sigma\) is a fan with \(\mathcal{R}(\Sigma) = \mathcal R_n\). For \(j \in I_n\), \cref{prop:rayset} implies \(\rho_{e_j^*}\) has the form \(\sum_{k=1}^n a_k e_k\) with \(a_k \in \{0, -1\}\) and \(a_j = -1\). By \cref{prop:rootset}, its number of associated roots is:
		\[
		|\mathcal R(\rho_{e_j^*})| = |\supp(\rho_{e_j^*})| \cdot (n + 1 - |\supp(\rho_{e_j^*})|).
		\]
		
		If \(\rho_{e_j^*} \notin \{-e_j, e_0\}\), \cref{rmk:rayset} implies \(|R(\rho_{e_j^*})| > n\). Since \(\Sigma\) is complete, \cref{lem:rayspan} gives at least \(n\) additional rays besides \(\rho_{e_j^*}\). Each such ray has at least \(n\) associated roots (\cref{rmk:rayset}). As roots uniquely correspond to rays, these \(n\) rays account for at least \(n^2\) roots. Thus:
		\[
		|\mathcal{R}(\Sigma)| \geq n^2 + |\supp(\rho_{e_j^*})| (n + 1 - |\supp(\rho_{e_j^*})|) > n^2 + n = |\mathcal R_n|,
		\]
		a contradiction. Hence \(\rho_{e_j^*} \in \{-e_j, e_0\}\) for each \(j \in I_n\).
		
		Suppose \(\rho_{e_j^*} = -e_j\) for some \(j\). Since \(\langle e_0, e_j^* \rangle = -1\), \cref{def:demazureroot} forces \(e_0 \notin \Sigma(1)\). Thus \(\rho_{e_j^*} = -e_j\) for all \(j \in I_n\), giving \(\{-e_i \mid i \in I_n\} \subseteq \Sigma(1)\). Note that \(\rho_{\alpha}\) for \(\alpha = e_{st}^*\) lies in this set, but \(\rho_{-e_i^*}\) does not. Thus the only roots \(-e_i^*\) are the only roots that lack an associated ray. Since each ray requires \(\geq n\) roots, we need a ray lying in
		\[
		\bigcap_{i=1}^n \mathcal{P}(-e_i^*) = \{-e_0\}.
		\]
		Thus \(\Sigma(1) = \{-e_i \mid i \in I_n^0\} = -\Sigma_{\P^n}(1)\).
		
		Now suppose \(\rho_{e_j^*} = e_0\) and let us prove \(\rho_{-e_j^*} = e_j\). By symmetry, \(\rho_{-e_j^*} \in \{e_j, -e_0\}\). Assume \(\rho_{-e_j^*} = -e_0\). Then \(-e_0 \in \Sigma(1)\), associated to roots \(-e_k^*\) for all \(k \in I_n\).
		
		We show that no element of \(\mathcal{P}(e_{jk}^*)\) is compatible with this ray configuration. By \cref{prop:rayset}, every \(\rho \in \mathcal{P}(e_{jk}^*)\) satisfies \(\langle \rho, e_j^* \rangle = -1\) or \(\langle \rho, -e_k^* \rangle = -1\). If \(\rho \in \Sigma(1)\), \cref{def:demazureroot} implies \(\rho = \rho_{e_j^*}\) or \(\rho = \rho_{-e_k^*}\). But \(\rho_{e_j^*} = e_0\) and \(\rho_{-e_k^*} = -e_0\) are not in \(\mathcal{P}(e_{jk}^*)\) since:
		\begin{enumerate}
			\item \(e_0 = -\sum_i e_i\) violates the conditions for \(e_{jk}^*\).
			
			\item \(-e_0\) similarly fails the pairing conditions.
		\end{enumerate}
		
		Thus no candidate ray exists for \(e_{jk}^*\), contradicting \(e_{jk}^* \in \mathcal{R}(\Sigma)\). Therefore \(\rho_{-e_j^*} \neq -e_0\), so \(\rho_{-e_j^*} = e_j\). Symmetric reasoning yields \(\Sigma(1) = \Sigma_{\P^n}(1)\).
	\end{proof}
	
	A direct consequence of this theorem and the framework developed in the previous section is the following corollary.
	
	\begin{corollary} \label{cor 1}
		Let \(X\) be a complete toric variety with \(\Aut^\circ X = \operatorname{PGL}_{n+1}\). Then \(X \simeq \P^n\).
	\end{corollary}
	
	\begin{proof}
		Let \(\Sigma_X\) be the fan corresponding to \(X\). By \cref{lem:rootdatum_demazureroots}, the roots of \(\Sigma_X\) coincide with the roots of \(\Aut^\circ_X = \operatorname{PGL}_{n+1}\). We know that the root system of \(\operatorname{PGL}_{n+1}\) is \(\mathcal{R}_n\). Thus the Demazure root set of \(\Sigma_X\) is \(\mathcal{R}_n\). 
		
		Since \(X\) is complete, \(\Sigma_X\) is a complete fan. By \cref{thm:roots_determine_projective_space}, we may assume without loss of generality that \(\Sigma_X(1) = \Sigma_{\P^n}(1)\). \cref{prop:projective_space_skeleton} then implies \(X \simeq \P^n\).
	\end{proof}
	\subsection{The general case}
	
	We now consider semisimple \(\Aut^\circ X\). By \cite[Section 3, \S 4, Corollary to Theorem 2]{DemazureToric}, \(\Aut^\circ_X\) decomposes as a product of projective linear groups. Combining results from the last subsection with combinatorial tools from Section 2, we establish propositions leading to \cref{thm:application}. The proof strategy proceeds as follows. For each \(i\), let \(T_i\) be the standard maximal torus of \(\operatorname{PGL}_{n_i+1}\), with character lattice \(M_i\) and cocharacter lattice \(N_i\). Then \(X\) is a toric variety with:
	\begin{enumerate}
		\item Torus \(T = T_1 \times \cdots \times T_k\).
		
		\item Character lattice \(M = \bigoplus_{i=1}^k M_i\).
		
		\item Cocharacter lattice \(N = \bigoplus_{i=1}^k N_i\)
	\end{enumerate}
	
	Let \(\Sigma\) be the fan of \(X\) in \(V \coloneqq N \otimes_{\Z} \R = \bigoplus_{i=1}^k V_i\), and \(\mathcal R(\Sigma) \subseteq M\) its Demazure root set. Then \(\mathcal R(\Sigma)\) must coincide with the product of root data:
	\begin{equation} \label{Eq 4.3}
	\mathcal	R(\Sigma) = \bigsqcup_{i=1}^k \mathcal R(\operatorname{PGL}_{n_i+1}),
	\end{equation}
	where \(\mathcal R(\operatorname{PGL}_{n_i+1}) \subseteq M_i\) is embedded into \(M\).
	
	We now derive consequences for \(\Sigma_X\). Our objective is to prove that for each \(i \in I_k\), the restricted fan \(\Sigma|_{V_i}\) equals \(\Sigma_{\P^{n_i}}\). This will follow from two subsequent lemmas. Let \(\rho \in \Sigma_X(1)\). Since \(\rho \neq 0\), there exists some \(t \in I_k\) such that \(\pi_t(\rho) \neq 0\), where \(\pi_t: N \to N_t\) is the projection. The following lemma explicitly describes \(\pi_t(\rho)\).
	
	\begin{lemma} \label{lem 1}
		For \(t \in I_k\), there exists a \(\mathbb{Z}\)-basis of \(N_t\) such that \(\pi_t(\Sigma_X(1)) \setminus \{0\} = \Sigma_{\P^{n_t}}(1)\).
	\end{lemma}
	
	We continue with a lemma ensuring \(\rho\) lies entirely in one factor:
	
	\begin{lemma} \label{lem 3}
		If \(\rho \in \Sigma_X(1)\) satisfies \(\pi_t(\rho) \neq 0\) for some \(t \in I_k\), then \(\rho \in N_t\).
	\end{lemma}
	
	In particular, \(\pi_t(\rho) = \rho\), so \(\rho \in \Sigma_{\P^{n_t}}(1)\). By \cref{lem 1} and \cref{lem 3}:
	\[
	\Sigma_X(1) = \bigsqcup_{t=1}^k \Sigma_X(1)|_{N_t} = \bigsqcup_{t=1}^k \Sigma_{\P^{n_t}}(1).
	\]
	By \cref{prop:projective_space_skeleton}, there is a unique complete toric variety with this 1-skeleton. \cref{thm:main} then implies:
	\[
	\Sigma_X = \bigoplus_{t=1}^k \Sigma_{\P^{n_t}}.
	\]
	Thus:
	\[
	X = \prod_{t=1}^k \P^{n_t},
	\]
	completing the proof of \cref{thm:application}.
	
	The remainder of this section proves \cref{lem 1} and \cref{lem 3}. Recall that \(X\) is a complete toric variety with \(\Aut^\circ X = \prod_{i=1}^k \operatorname{PGL}_{n_i+1}\), where \(T_i\) is the standard maximal torus of \(\operatorname{PGL}_{n_i+1}\), and \(M_i\), \(N_i\) are its character and cocharacter lattices.
	
	First, we prove \cref{lem 1}, as it will later help us establish \cref{lem 3}. \cref{Eq 4.3} describes the roots of \(\Aut_X\) lying in \(M_t\). When these roots are interpreted as elements of a fan in \(N_t\), they constrain the rays of \(\Sigma_X(1)\) that project nontrivially onto this space. We now proceed with the proof.
	
	\begin{proof}[Proof of \cref{lem 1}]
		Let \(t \in I_k\). Using \cref{Eq 4.3}, we observe that in a suitable \(\mathbb{Z}\)-basis, the roots of \(\text{PGL}_{n_{t+1}}\) in \(M_t\) are:
		\begin{equation} \label{Eq 4.4}
			\{e^*_{i,t}\} \cup \{e^*_{ij,t} \mid i \neq j\}. 
		\end{equation}
		This fixes a \(\mathbb{Z}\)-basis \(\{e_{i,t}\}\) of \(N_t\). If \(\pi_t(\rho) \neq 0\) for \(\rho \in \Sigma_X(1)\), then
		\[
		\pi_t(\rho) = \sum_{i=1}^{n_t} a_i e_{i,t}.
		\]
		Since \(\pi_t(\rho) \neq 0\), there exists \(\ell \in I_{n_t}\) with \(a_\ell \neq 0\). Without loss of generality, assume \(a_\ell > 0\). Note that \(\langle \rho, -e^*_{\ell,t} \rangle = -a_\ell < 0\). As \(-e^*_{\ell,t}\) is a root in \(M_t\), \cref{def:demazureroot} implies that \(\rho\) must be \(\rho_{-e^*_{\ell,t}}\). Therefore,
		\begin{equation}\label{Eq 4.5}
			\pi_t(\Sigma_X(1)) \setminus \{0\} = \{\pi_t(\rho_\alpha) \mid \alpha \text{ is a root in } M_t\}.
		\end{equation}
		For any \(\alpha\) in this set, \cref{def:demazureroot} guarantees a ray \(\rho_\alpha \in \Sigma_X(1)\) such that \(\langle \rho_\alpha, \alpha \rangle = -1\). Since \(\alpha \in M_t\), we have \(\langle \rho_\alpha, \alpha \rangle = \langle \pi_t(\rho_\alpha), \alpha \rangle\). Thus, the set (\ref{Eq 4.5}) consists of rays in \(N_t\) associated with the roots (\ref{Eq 4.4}), which are the roots of \(\PGL_{n_{t+1}}\). By \cref{thm:roots_determine_projective_space}, this set of rays must be the \(1\)-skeleton of \(\mathbb{P}^{n_t}\).
	\end{proof}
	
	The following lemma will be used to prove \cref{lem 3}. It exploits the uniqueness of rays associated with roots to show that the projections \(\pi_t\) are \textit{injective} when restricted to rays of the fan.
	
	\begin{lemma}\label{lem 2}
		Let \( t \in I_k \). If \( \rho \in \pi_t(\Sigma_X(1)) \setminus \{0\} \), then \( |\pi_t^{-1}(\rho) \cap \Sigma_X(1)| = 1 \).
	\end{lemma}
	
	\begin{proof}
		Let \(\rho_1, \rho_2 \in \Sigma_X(1)\) satisfy \(\pi_t(\rho_1) = \pi_t(\rho_2) \neq 0\). By Proposition 4.10, \(\pi_t(\rho_1)\) is a ray of the \(1\)-skeleton of \(\Sigma_{\mathbb{P}^{n_t}}\). By assumption, \(\pi_t(\rho_2)\) is the same ray. Since every ray in \(\Sigma_{\mathbb{P}^{n_t}}(1)\) has associated roots, there exists some root \(\alpha \in M_t\) such that:
		\[
		\langle \rho_1, \alpha \rangle = \langle \pi_t(\rho_1), \alpha \rangle = -1.
		\]
		The same holds for \(\rho_2\) with the same root \(\alpha\). By \cref{def:demazureroot}, there is a unique ray in \(\Sigma_X(1)\) satisfying this property. Thus \(\rho_1 = \rho_2\).
	\end{proof}
	
	\begin{proof}[Proof of \cref{lem 3}]
		Without loss of generality, assume \(n_s \geq n_t\). We proceed by contradiction. Suppose there exist \(\rho \in \Sigma_X(1)\) and distinct \(s, t \in I_k\) such that \(\pi_s(\rho) \neq 0\) and \(\pi_t(\rho) \neq 0\). By \cref{lem 1}, the projections of \(\rho\) onto \(N_s\) and \(N_t\) are rays of \(\Sigma_{\mathbb{P}^{n_s}}(1)\) and \(\Sigma_{\mathbb{P}^{n_t}}(1)\), respectively. After fixing suitable \(\mathbb{Z}\)-bases for \(N_s\) and \(N_t\), we may assume \(\pi_s(\rho) = e_{1,s}\) and \(\pi_t(\rho) = e_{1,t}\). Using dual bases, the roots of \(\text{PGL}_{n_s+1}\) in \(M_s\) and \(\text{PGL}_{n_t+1}\) in \(M_t\) are respectively:
		\[
		\{e_{i,s}^*\} \cup \{e_{ij,s}^* \mid i \neq j\} \quad \text{and} \quad \{e_{i,t}^*\} \cup \{e_{ij,t}^* \mid i \neq j\}.
		\]
		We consider two cases: first when \(n_s = n_t = 1\), and then when \(n_s > 1\).
		
		\textit{Case 1: \(n_s = n_t = 1\)}. By \cref{lem 1}, \(\pi_s(\Sigma_X(1)) \setminus \{0\} = \Sigma_{\mathbb{P}^1}(1)\) and \(\pi_t(\Sigma_X(1)) \setminus \{0\} = \Sigma_{\mathbb{P}^1}(1)\). We now analyze subcases:
		
		\textit{Subcase 1:} There exists \(\rho' \in \Sigma_X(1)\) such that \(\pi_{st}(\rho') = -e_{1,s} - e_{1,t}\), where \(\pi_{st}: N \to N_s \oplus N_t\) is the projection. In this case, note that by \cref{lem 2}, for every \(\hat{\rho} \in \Sigma_X(1) \setminus \{ \rho, \rho' \}\), we have \(\pi_{st}(\hat{\rho}) = 0\), since the images of \(\rho\) and \(\rho'\) under \(\pi_s\) and \(\pi_t\) constitute all nonzero rays of \(\pi_s(\Sigma_X(1))\) and \(\pi_t(\Sigma_X(1))\), respectively. Thus, only two rays of \(\Sigma_X(1)\) project nontrivially to \(N_s \oplus N_t\).  
		
		Since the image of any cone \(\sigma \in \Sigma_X\) under \(\pi_{st}\) is generated by the images of its rays, every cone \(\sigma \in \Sigma_X\) satisfies \(\sigma \subset \pi_{st}^{-1}(\operatorname{cone}(\pi_{st}(\rho), \pi_{st}(\rho')))\). Consequently, for any \(x \in N \setminus \pi_{st}^{-1}(\operatorname{cone}(\pi_{st}(\rho), \pi_{st}(\rho')))\), we have \(x \notin \sigma\) for all \(\sigma \in \Sigma_X\). This contradicts the completeness of \(\Sigma_X\), as the fan fails to cover the ambient space. Thus, this subcase cannot occur.
		
		\textit{Subcase 2:} There exist \(\rho', \rho_3 \in \Sigma_X(1)\) such that \(\pi_{st}(\rho') = (-e_{1,s}, 0)\) and \(\pi_{st}(\rho_3) = (0, -e_{1,t})\) in \(N_s \oplus N_t\). By \cref{lem 2}, for every \(\eta \in \Sigma_X(1) \setminus \{\rho, \rho', \rho_3\}\), we have \(\pi_{st}(\eta) = (0,0)\). Define \(\alpha := -e_{1,s}^* + e_{1,t}^*\). We now prove \(\alpha\) is a Demazure root of \(\Sigma_X\). First, note that \(\alpha\) is the sum of two Demazure roots of \(\Sigma_X\). Direct computation shows:
		\[
		\langle \rho_3, \alpha \rangle = -1, \quad \langle \rho, \alpha \rangle = 0, \quad \langle \rho', \alpha \rangle = 1.
		\]
		For any \(\eta \in \Sigma_X(1) \setminus \{\rho, \rho', \rho_3\}\):
		\[
		\langle \eta, \alpha \rangle = \langle \eta, -e_{1,s}^* \rangle + \langle \eta, e_{1,t}^* \rangle.
		\]
		Since \(-e_{1,s}^*\) and \(e_{1,t}^*\) are Demazure roots of \(\Sigma_X\), and \(\eta\) is not the distinguished ray for either root (as \(\rho\) pairs to \(-1\) with \(-e_{1,s}^*\) and \(\rho_3\) pairs to \(-1\) with \(e_{1,t}^*\)), we have:
		\[
		\langle \eta, -e_{1,s}^* \rangle \geq 0 \quad \text{and} \quad \langle \eta, e_{1,t}^* \rangle \geq 0.
		\]
		Thus \(\langle \eta, \alpha \rangle \geq 0\). Therefore, \(\alpha\) satisfies the definition of a Demazure root with distinguished ray \(\rho_3\). However, \(\alpha \notin \bigsqcup_{i=1}^{k} \mathcal{R}(\mathrm{PGL}_{n_i+1})\), contradicting the hypothesis that all Demazure roots belong to this disjoint union.
		
		We conclude that if \( n_s = n_t = 1 \), there cannot exist \(\rho \in \Sigma_X(1)\) with both \(\pi_s(\rho) \neq 0\) and \(\pi_t(\rho) \neq 0\). We now proceed to the case where \( n_s > 1 \). Following the strategy of Subcase 2, we construct a root outside the root system at $X$. Observe that \(\rho\) is associated with roots in the set:
		\[
		\{e_{k,t}^* \mid k \in I_{n_t}\} \cup \{-e_{1,t}^*\},
		\]
		since for every \(\alpha_t\) in this set \(\langle \rho, \alpha_t \rangle = -1.\) We now justify the existence of \(\rho' \in \Sigma_X(1) \setminus \{\rho\}\) and a Demazure root \(\alpha_s \in M_s\) satisfying:
		\begin{equation} \label{Eq 4.6}
			\langle \rho, \alpha_s \rangle = 1, \quad \langle \rho', \alpha_s \rangle = -1, \quad \langle \rho', \alpha_t \rangle = 0.
		\end{equation}
		Take \(\rho' \in \Sigma_X(1)\) with \(\pi_s(\rho') \neq 0\). Either \(\pi_t(\rho') = 0\) or \(\pi_t(\rho') \neq 0\). First assume \(\pi_t(\rho') = 0\). By \cref{lem 1}, we consider the following subcases:
		
		\textit{Subcase 1:} \(\pi_s(\rho') = e_{j,s}\) for some \(j \in I_n\) (note that by \cref{lem 2}, \(j \neq 1\)). In this case, we choose \(\alpha_s = e_{1j,s}^*\) and \(\alpha_t = -e_{1,t}^*\). Then \(\rho\), \(\alpha_t\), \(\rho'\), and \(\alpha_s\) satisfy condition (\ref{Eq 4.6}).
		
		\textit{Subcase 2:} \(\pi_s(\rho') = e_{0,s}\). In this case, we choose \(\alpha_s = e_{1,s}^*\) and \(\alpha_t = -e_{1,t}^*\). Then \(\rho\), \(\alpha_t\), \(\rho'\), and \(\alpha_s\) satisfy condition (\ref{Eq 4.6}). Now assume \(\pi_t(\rho') \neq 0\). By \cref{lem 1}, we consider the following subcases:
		
		\textit{Subcase 1:} \(\pi_t(\rho') = e_{j,t}\) for some \(j \in I_n\) (note \(j \neq 1\) by \cref{lem 2}). We choose \(\alpha_t = -e_{1,t}^*\). As before, \(\pi_s(\rho') = e_{j',s}\) for some \(j' \in I_n\) or \(\pi_s(\rho') = e_{0,s}\). For each possibility, select \(\alpha_s\) as in the previous subcases. Then \(\rho\), \(\alpha_t\), \(\rho'\), and \(\alpha_s\) satisfy (\ref{Eq 4.6}).
		
		\textit{Subcase 2:} \(\pi_t(\rho') = e_{0,t}\). Here, no root \(\alpha_t \in M_t\) satisfies (\ref{Eq 4.6}). However, since \(n_s > 1\), \(|\Sigma_{\mathbb{P}^{n_s}}(1)| \geq 3\). Thus there exists \(\tilde{\rho} \in \Sigma_X(1) \setminus \{\rho, \rho'\}\) with \(\pi_s(\tilde{\rho}) \neq \pi_s(\rho)\) and \(\pi_s(\tilde{\rho}) \neq \pi_s(\rho')\). By \cref{lem 2}, \(\pi_t(\tilde{\rho}) \neq e_{0,t}\), reducing this to the previous subcases.
		
		This proves we can always find \(\rho'\), \(\alpha_s\), and \(\alpha_t\) satisfying (\ref{Eq 4.6}). We now show \(\alpha_t + \alpha_s\) is a Demazure root of \(\Sigma_X\). First,  
		\[
		\langle \rho', \alpha_t + \alpha_s \rangle = -1 \quad \text{and} \quad \langle \rho, \alpha_t + \alpha_s \rangle = 0.
		\]  
		Since \(\alpha_t\) and \(\alpha_s\) are Demazure roots with distinguished rays \(\rho\) and \(\rho'\) respectively, \(\langle \eta, \alpha_t \rangle \geq 0\) and \(\langle \eta, \alpha_s \rangle \geq 0\) for all \(\eta \in \Sigma_X(1) \setminus \{\rho, \rho'\}\). Thus \(\langle \eta, \alpha_t + \alpha_s \rangle \geq 0\). Therefore \(\alpha_t + \alpha_s\) is a Demazure root of \(\Sigma_X\), contradicting the hypothesis that \(\bigsqcup_{i=1}^k \mathcal{R}(\mathrm{PGL}_{n_i+1})\) contains all Demazure roots of \(\Sigma_X\).
	\end{proof}
%Bibliography

\bibliographystyle{alpha}
\bibliography{bibliografia.bib}

\begin{thebibliography}{Hum75}

\bibitem[Bou08]{BourbakiLie46}
Nicolas Bourbaki.
\newblock {\em Elements of mathematics. {Lie} groups and {Lie} algebras. {Chapters} 4--6. {Transl}. from the {French} by {Andrew} {Pressley}}.
\newblock Berlin: Springer, paperback reprint of the hardback edition 2002 edition, 2008.

\bibitem[CLS11]{CoxToric}
David~A. Cox, John~B. Little, and Henry~K. Schenck.
\newblock {\em Toric varieties}, volume 124 of {\em Grad. Stud. Math.}
\newblock Providence, RI: American Mathematical Society (AMS), 2011.

\bibitem[Dem70]{DemazureToric}
Michel Demazure.
\newblock Sous-groupes alg{\'e}briques de rang maximum du groupe de {Cremona}.
\newblock {\em Ann. Sci. {\'E}c. Norm. Sup{\'e}r. (4)}, 3:507--588, 1970.

\bibitem[Ful93]{FultonToric}
William Fulton.
\newblock {\em Introduction to toric varieties. {The} 1989 {William} {H}. {Roever} lectures in geometry}, volume 131 of {\em Ann. Math. Stud.}
\newblock Princeton, NJ: Princeton University Press, 1993.

\bibitem[Hum75]{HumphLAG}
James~E. Humphreys.
\newblock {\em Linear algebraic groups}, volume~21 of {\em Grad. Texts Math.}
\newblock Springer, Cham, 1975.

\bibitem[Kur10]{Shintaro}
Shintar{\^o} Kuroki.
\newblock Characterization of homogeneous torus manifolds.
\newblock {\em Osaka J. Math.}, 47(1):285--299, 2010.

\bibitem[LLA22]{LucLieToric}
Alvaro Liendo and Giancarlo Lucchini~Arteche.
\newblock Automorphisms of products of toric varieties.
\newblock {\em Mathematical Research Letters}, 29(2):529–540, 2022.

\bibitem[MO67]{MatOo}
Hideyuki Matsumura and Frans Oort.
\newblock Representability of group functors, and automorphisms of algebraic schemes.
\newblock {\em Invent. Math.}, 4:1--25, 1967.

\bibitem[Spr98]{SpringerLAG}
Tonny~A. Springer.
\newblock {\em Linear algebraic groups.}, volume~9 of {\em Prog. Math.}
\newblock Boston, MA: Birkh{\"a}user, 2nd ed. edition, 1998.

\end{thebibliography}
\end{document}